\documentclass[a4paper,12pt,reqno]{amsart}
\usepackage{latexsym}
\usepackage{amssymb} 
\usepackage{mathrsfs}
\usepackage{amsmath}
\usepackage{latexsym}

\usepackage{delarray}
\usepackage{amssymb,amsmath,amsfonts,amsthm,mathrsfs}
\usepackage{blkarray}

\usepackage{hyperref}
\renewcommand\eqref[1]{(\ref{#1})} 
 \newtheorem{theorem}{Theorem}[section]
 \newtheorem{cor}[theorem]{Corollary}
 
 \newtheorem{lemma}[theorem]{Lemma}
 \newtheorem{proposition}[theorem]{Proposition}
 \theoremstyle{definition}
 \newtheorem{defn}[theorem]{Definition}
 \theoremstyle{remark}
 \newtheorem{rem}[theorem]{Remark}
 \newtheorem{remark}[theorem]{Remark}
 
 \numberwithin{equation}{section}
 \newcommand{\half}{\frac{1}{2}}

\newcommand{\ene}{\mathbb{N}}
\newcommand{\ar}{\mathbb{R}}
\newcommand{\er}{\mathbb{R}}

\newcommand{\To}{\mathbb{T}}
\newcommand{\arn}{{\mathbb{R}}^n}
\newcommand{\Rn}{{\mathbb{R}}^n}

\newcommand{\partvv}{\frac{\partial ^2 v}{\partial t^2}}

\newcommand{\partu}{\frac{\partial u}{\partial t}}
\newcommand{\partv}{\frac{\partial v}{\partial t}}
\newcommand{\partuu}{\frac{\partial ^2 u}{\partial t^2}}

\newcommand{\partt}{\frac{\partial}{\partial t}}
\newcommand{\bi}{\begin{itemize}}
\newcommand{\ei}{\end{itemize}}
\newcommand{\be}{\begin{enumerate}}
\newcommand{\ee}{\end{enumerate}}
\newcommand{\beq}{\begin{equation}}
\newcommand{\eq}{\end{equation}}

\newcommand{\zet}{\mathbb{Z}}
\newcommand{\Co}{\mathbb{C}}
\newcommand{\zn}{\mathbb{Z}^n}
\newcommand{\tn}{\mathbb{T}^n}
\newcommand{\Tn}{\mathbb{T}^n}
\newcommand{\fou}{\mathcal{F}}
\newcommand{\ese}{\mathcal{S}}

\newcommand{\jpxi}{\langle \xi \rangle}
\def\Op{\text{\rm Op}}

\begin{document}

\title[Well-posedness for a class of pseudo-differential hyperbolic  equations]{Well-posedness for a class of pseudo-differential hyperbolic equations on the torus}

 \author[D. Cardona]{Duv\'an Cardona}
\address{
 Duv\'an Cardona:
  \endgraf
  Department of Mathematics: Analysis, Logic and Discrete Mathematics
  \endgraf
  Ghent University, Belgium
  \endgraf
  {\it E-mail address} {\rm duvan.cardonasanchez@ugent.be}
  }
 \author[J. Delgado]{Julio Delgado}
\address{
  Julio Delgado:
  \endgraf
  Departmento de Matematicas
  \endgraf
  Universidad del Valle
  \endgraf
  Cali-Colombia
  \endgraf
    {\it E-mail address} {\rm delgado.julio@correounivalle.edu.co}}
  
\author[M. Ruzhansky]{Michael Ruzhansky}
\address{
  Michael Ruzhansky:
  \endgraf
  Department of Mathematics: Analysis, Logic and Discrete Mathematics
  \endgraf
  Ghent University, Belgium
  \endgraf
 and
  \endgraf
  School of Mathematical Sciences
  \endgraf
  Queen Mary University of London
  \endgraf
  United Kingdom
  \endgraf
  {\it E-mail address} {\rm michael.ruzhansky@ugent.be}
  }

\thanks{The authors are supported by the FWO Odysseus 1 grant G.0H94.18N:Analysis and Partial Differential Equations and by the Methusalem programme of the Ghent University Special Research Fund(BOF)(Grant number 01M01021). D. Cardona is also supported
by the Research Foundation-Flanders (FWO) under the postdoctoral grant No 1204824N. J. Delgado is also supported by Vice. Inv. Universidad del Valle Grant CI 71329, MathAmSud and Minciencias-Colombia under the project MATHAMSUD21-MATH-03. Michael Ruzhansky is also supported by EPSRC grants EP/R003025/2 and EP/V005529/1.}

\subjclass[2020]{Primary {35S10, 35S05; Secondary 35R11, 35A01}.}

\keywords{Fractional hyperbolic equations,   periodic pseudo-differential operators, energy estimates, microlocal analysis.}

\date{\today}

\begin{abstract}In this paper we establish the well-posedness of the Cauchy problem for a class of pseudo-differential hyperbolic equations on the torus. The class considered here includes a space-like fractional order 
 Laplacians. By applying the toroidal pseudo-differential calculus 
 we establish  regularity  estimates, existence and uniqueness in the scale of the  standard Sobolev spaces on the torus.\end{abstract}

\maketitle
\tableofcontents
\allowdisplaybreaks
\section{Introduction}

\subsection{Outline}
In this work we analyse the well-posedness for  hyperbolic periodic problems associated to positive elliptic pseudo-differential operators on the torus $\tn\cong [0,1)^n$. Although the analysis carried out here can be extended to general compact Lie groups, here we concentrate our attention to the commutative case, namely, the case of the torus, due to the simplified formulation of the pseudo-differential calculus in this context, established in terms of the periodic Fourier analysis as developed in \cite{rt:torus,rt:book}. The case of general compact Lie groups will be considered in a subsequent work with the periodic case as a fundamental model.%

To precise the notation,  let $P(x,D)$ be an elliptic positive pseudo-differential operator of order $\nu>0$ on the torus.  In terms of the periodic pseudo-differential calculus developed in \cite{rt:torus,rt:book}, to the operator $P(x,D)$  one  can associate a  symbol $p:=p(x,\xi)$ globally defined on the phase space $\tn\times \zn,$ allowing the integral representation
\begin{align}\label{toroidal:quantisation}
    P(x,D)f(x)=\sum_{\xi\in \zn }\int\limits_{ \tn}e^{2\pi i (x-y)\cdot \xi} p(x,\xi)f(y)dy,\,f\in C^\infty(\tn).
\end{align}We have used the notation $ x\cdot \xi=x_1\xi_1+\cdots +x_n\xi_n$ for the standard inner product.
In this setting the ellipticity condition means that the  symbol satisfies the estimate
\begin{equation}\label{ellipticity:intro}
    C_1 \lvert\xi \rvert^\nu \leq  \lvert p(x,\xi) \rvert \leq C_2 \lvert \xi \rvert^\nu, \,\xi\in \zn, \xi\neq 0,\,x\in \tn,
\end{equation}  for some positive constants $C_1,C_2>0.$ Then, with the notation above, for a given time $T>0,$ our main goal is to establish the well-posedness for the following Cauchy problem : 
\begin{equation}
 \left\{
\begin{array}{rl}
{\displaystyle \partvv} =&-P(x,D)v+\omega, \,\, \mbox{ (in the sense of   }\mathcal{D}'(]0,T[\times\mathbb{T}^n))\\
{\displaystyle v}(0)=&f,\\
{\displaystyle \partv}(0)=&g,\\
\end{array} \right.
\label{probei2gax}\end{equation}
where  $f\in H^{s}(\tn),\, g\in  H^{s-\frac{\nu}{2}}(\tn),$  and $\omega\in L^2([0,T], H^{s-\frac{\nu}{2}}(\tn))$,  for some $s\in\ar$.
We clarify the contributions of this note in Theorem \ref{main2:2} below,   but first, we discuss the main differences between the  Euclidean diffusion operators and the ones in the periodic setting. For instance, a first look at the models of the form \eqref{probei2gax} could give the impression that their analysis can be carried out by the standard periodisation techniques, however, this is not the case as we will explain in the simplest case of the fractional Laplacian on the torus. Then, the situation changes dramatically if one considers non-local perturbations of this operator even of lower order as well as perturbations with variable coefficients.  For the main aspects about the periodisation of Euclidean models we refer the reader to \cite{rt:torus,rt:book}.

\subsection{Periodic vs Euclidean models}\label{discussion}
Our setting includes the case of the fractional Laplacian $P(x,D)=(-\Delta)^{\frac{\nu}{2}}.$ Let us observe that the toroidal quantization formula in \eqref{toroidal:quantisation} and the pseudo-differential calculus developed in \cite{rt:torus,rt:book} will allow us to consider the diffusion operator $(-\Delta)^{\frac{\nu}{2}}$ as a pseudo-differential operator for any  $\nu>0.$ Indeed, the functional calculus for the Laplacian on the torus implies the integral  formula
\begin{align}\label{toroidal:fraclap}
    (-\Delta)^{\frac{\nu}{2}}f(x)=\sum_{\xi\in \zn }\int_{ \tn}e^{2\pi i (x-y)\cdot \xi} (2\pi)^{\nu} \lvert\xi \rvert^\nu f(y)dy,\,f\in C^\infty(\tn).
\end{align} From this formula one can deduce the non-locality of the operator if e.g. $\nu\notin 2\mathbb{N}_0.$ This ``global pseudo-differential representation''
 is not available for the fractional Laplacian on $\Rn$, where it is only a pseudo-diferential operator when $\nu$ is an even integer, due to the lack of smoothness of the symbol at the origin.

One reason to emphasize about the role of the fractional Laplacian on the torus as a crucial model in the analysis of the hyperbolic problems as in \eqref{probei2gax} came from the following fact. One may think that the analysis of non-local operators like the power $(-\Delta)^{\frac{\nu}{2}}$ can be obtained by periodisation techniques of
the results for the fractional Laplacian on $\arn$. However, the analysis of periodic models changes dramatically due to the non-locality of the operators. Another obstruction came with the structure of distributions on the torus. Indeed, observe that the class of  smooth functions, as well as the class of $L^p$-functions on $\tn $, cannot be identified with the Schwartz class or with the $L^p$-spaces on $\arn,$ respectively, even with assumptions of extensions from the torus to the real line under the periodicity condition.

On the other hand, a more distinctive fact arises with the behavior of the singularities of conormal distributions. Indeed,  the kernel of the fractional Laplacian on the Euclidean space is not an integrable function. Then, as it was pointed out in  \cite{ron:t},  its “periodisation”
in principle has just a formal meaning. It is not clear that the fractional Laplacian acting on
periodic functions coincides with the fractional
Laplacian on the torus as defined as a Fourier multiplier of Fourier series. We refer the reader to Roncal and Stinga \cite{ron:t} for the analysis of the fractional Laplacian on the torus and for its properties in comparison with the ones of the fractional Laplacian on the Euclidean space.

\subsection{State-of-the-art} 
The research in evolution equations governed by the fractional Laplacian and other elliptic pseudo-differential operators  on different structures  has been intensive in the last decades. However the literature on the fractional hyperbolic case is less known compared with the ones  involving models of fractional diffusion.  Fractional hyperbolic equations have been treated for instance in  \cite{del:deghyp1} within the framework of Weyl-H\"ormander calculus. We refer the reader to Remark \ref{extension:remark} for a discussion about hyperbolic problems for  pseudo-differential terms in the Euclidean space as well as in the periodic setting.  For the analysis of boundary value problems for the Euclidean fractional Laplacian we refer to Grubb \cite{GGr:aq}.

On the other hand, the wide variety of applications of the fractional diffusion is spread throughout fluid mechanics \cite{pc:lf}, \cite{mb:lf}, mathematical finance  \cite{c:ff},  fractional dynamics \cite{skb:kk}, \cite{nl:aw},  strange kinetics and anomalous transport, see  \cite{Kiselev:Nazarov},\cite{dab:geo},\cite{caf:annals} and the references therein. Also, drift-diffusion equations with fractional diffusion  have intensely attracted the interest during  the last 12 years starting with the works of  Caffarelli and Vasseur in \cite{caf:annals} and their subsequent developments.
In a particular but crucial setting, the fractional Laplacian, which is the model of the operators considered in this work, is an interesting object in its own, and on the torus it has been studied in  \cite{ron:t}. Several aspects of the harmonic analysis of the fractional Laplacian on lattices also have been investigated in \cite{ron:di} and in   \cite{ron:ddf}. Recent works on the fractional Laplacian and its different generalizations can be found in \cite{rup:f1}, \cite{Cao:hy}, \cite{Uhlm:fl}, \cite{GGr:aq}, \cite{val:p2}, \cite{val:p3} and the references therein.   An accessible  presentation of the fractional Sobolev spaces and the  fractional Laplacian can be found in \cite{val:p1}.

\subsection{Main result  and organisation of the paper}
We will apply the point of view of the global quantisation for pseudo-differential operators 
 on the torus as developed in  \cite{rt:torus,rt:book}. In order to present our main result we will employ the following additional notation.
 \begin{itemize}
\item In the periodic hyperbolic model \eqref{probei2gax}, the symbol $p:=p(x,\xi)$ of the operator will be considered in the H\"ormander classes $S^\nu_{\rho,\delta}(\tn\times \zn),$ see Section \ref{Preliminaries} for  the definition of these classes. 
\item The class of pseudo-differential operators $\Psi^\nu_{\rho,\delta}(\tn\times \zn)$ denotes the family of pseudo-differential operators with symbols in the toroidal class $S^\nu_{\rho,\delta}(\tn\times \zn).$ We refer to this class as the toroidal/periodic $(\rho,\delta)$-H\"ormander class of order $\nu.$ 
 \item We will denote by $H^s(\tn)$ the standard Sobolev space of order $s\in \ar,$ on the torus $\tn,$ see Remark \ref{Remark:Sob} for details.
 \end{itemize}
Our main theorem can be stated in the following way.
 \begin{theorem}\label{main2:2}  Let $\nu>0$, a given time $T>0$,   and let $0\leq \delta<\rho\leq 1$. Let $s\in\ar$, $f_0\in H^s(\tn),\,f_1\in H^{s-\frac{\nu}{2}}(\tn)$ and let $\omega\in L^2([0,T], H^{s-\frac{\nu}{2}}(\tn))$. Let $P(x,D)\in \Psi^\nu_{\rho,\delta}(\tn\times \zn)$ be a positive elliptic pseudo-differential operator in the $(\rho,\delta)$-H\"ormander class of order $\nu$.

Then, there exists a unique
 solution $u\in C([0,T], H^s(\tn))$ of the Cauchy problem
 \begin{equation}
 \left\{
\begin{array}{rl}
{\displaystyle\partuu }\,\,\,\,=& -P(x,D)u+w, \,\, \mbox{ (in the sense of }\mathcal{D}'(]0,T[\times\Tn))\\
{\displaystyle u}(0)=&f_0,\\
{\displaystyle\partu}(0)=& f_1,
\end{array} \right.
\label{probgg2}\end{equation}  
satisfying the following energy estimate 
\begin{equation}\label{einw1:1c}
    \Vert u(t)\Vert_{H^s(\tn)}^2\leq Ce^{Ct}\left(\Vert f_0\Vert_{H^s(\tn)}^2+\Vert f_1\Vert_{H^{s-\frac{\nu}{2}}(\tn) }^2+\int\limits_0^t\Vert w(\tau)\Vert^2_{H^{s-\frac{\nu}{2}}(\tn)}d\tau\right).
\end{equation}
Consequently, if $w\in C^{\infty}([0,T]\times \tn)$ and every  $f_i\in C^{\infty}(\tn)$ is smooth,
   then the solution $u$  belongs to the class $C^{\infty}([0,T]\times \tn)$.
\end{theorem}
\begin{rem}\label{extension:remark} Fractional hyperbolic equations have been treated for instance in  \cite{del:deghyp1}, where in particular a version of Theorem \ref{main2:2} has been proved within the framework of the Weyl-H\"ormander calculus. In particular, such a result provides the well-posedness for hyperbolic equations in the Euclidean H\"ormander classes. However, in view of the discussion above,  our main Theorem \ref{main2:2} cannot be recovered from its Euclidean analogue just by periodisation techniques, see Subsection \ref{discussion} for details.     
\end{rem}
\begin{rem}
    In order to prove the energy estimate in Theorem \ref{main2:2}, the hyperbolic problem is simplified to a vector-valued diffusion model. In this context, we will employ a matrix-valued version of the toroidal calculus  developed in  \cite{rt:torus,rt:book}, see Subsection \ref{matrix:values:classes} for details. The energy estimate obtained from Theorem \ref{main2:2} applied to the fractional Laplacian on the  torus can be found in Corollary \ref{main2}.  In this case we use the property $ (-\Delta)^{\frac{\nu}{2}}\in \Psi^\nu_{1,0}(\tn\times \zn).$ However, other examples in the $(\rho,\delta)$-setting appear e.g. if $\delta=0$ with oscillating multipliers 
    \begin{equation}\label{osci}
        P(x,D)=(-\Delta)^{\frac{\nu}{2}}e^{i(-\Delta)^{\frac{1-\rho}{2}}}\in \Psi^\nu_{\rho,0}(\tn\times \zn),
    \end{equation}
     with the oscillating parameter $\rho$ satisfying the inequality $0< \rho\leq 1.$ 
\end{rem}
\begin{rem}
    We also observe that when localising the H\"ormander classes from the Euclidean space  to the torus one can construct the localised classes $\Psi^\nu_{\rho,\delta,loc}(\tn)$ but due to the structure of the asymptotic expansions on these classes, the conditions $0\leq \delta<\rho\leq 1,$ and $\rho\geq 1-\delta$ arise. In particular these two restrictions imply the inequality $\rho>1/2.$ Our  Theorem \ref{main2:2} operates below the range $\rho\leq 1/2,$ where the construction under local coordinate systems is not available. For instance, our approach allow the analysis of the operators in \eqref{osci} when  $\rho\leq 1/2.$
\end{rem}
 This work is organised as follows:
 \begin{itemize}
     \item in Section \ref{Preliminaries} we give a brief review of basic preliminaries on periodic pseudo-differential operators. 
     \item In Section \ref{Sect:3} we develop our analysis for the proof of our main result.  We first obtain energy estimates for  first-order systems, see Theorem \ref{order1}. Then, this analysis is used in the proof  of Theorem \ref{main2:2}.  
 \end{itemize}

\section[Pseudo-differential calculus  on the torus ]{Pseudo-differential operators on the torus}\label{Preliminaries}
In this section we review some basic elements of the pseudo-differential calculus on the torus. For a comprehensive study  on this theory we refer to  \cite{rt:torus,rt:book} and to \cite{rt:torus} for the quantisation on the torus. 
\subsection{Scalar-valued classes on the torus}
The use of  global representations by Fourier series instead of local representations in coordinate charts as it is more customary when working on manifolds has some advantages.  In particular, it will lead us to improve some results regarding the range of the exponent in the fractional Laplacian on the torus. 

For our further analysis let us fix the required notation. Let $\tn=\Rn/\zn$ be the $n$-dimensional torus. As it is well known,    
 the Schwartz space $\ese(\Rn)$ is convenient to define the Euclidean Fourier  transform. In the periodic setting, its counterpart is the corresponding  Schwartz space $\ese(\zn)$ with respect to the dual variable. We introduce it as follows. 
\begin{defn} Let  $\ese(\zn)$ denote the space of {\em rapidly decaying functions} $\varphi:\zn\rightarrow\Co$. That is, $\varphi\in\ese(\zn)$ if for any $M>0$ there exists a constant $C_{\varphi,\, M}$ such that
\[\forall\xi\in \mathbb{Z}^n,\,\,\lvert\varphi(\xi)\rvert\leq C_{\varphi,\, M}\jpxi^{-M}\]
for all $\xi\in\zn$. The topology on $\ese(\zn)$ is given by the seminorms $p_k$, where
$k\in\ene_0$ and  $p_k(\varphi):=\sup\limits_{\xi\in\zn}\jpxi^{k}\lvert\varphi(\xi)\rvert.$
\end{defn}
We can now define the toroidal or periodic Fourier transform. 
\begin{defn} We denote by $\fou_{\tn}$ the {\em toroidal Fourier transform}  
 $\fou_{\tn}:C^{\infty}(\tn)\rightarrow \ese(\zn),\,\, \phi\mapsto \widehat{\phi}\,\, $ by
\[(\fou_{\tn}\varphi)(\xi)=\int\limits_{\tn}e^{-i2\pi x\cdot \xi}\varphi(x)dx,\;\, \xi\in \zn.\]
\end{defn}
\begin{remark}Note that the notion of Fourier transform is related to the spectral decomposition of $L^2(\tn)$ with respect to the collection $\{e^{2\pi ix\cdot k}\}_{k\in\zn},$ which is an orthonormal basis of $L^2(\tn)$. Also,  we  sometimes write $\widehat{\phi}(\xi)$ instead of $(\fou_{\tn}\varphi)(\xi)$. The toroidal Fourier transform is a bijection and its inverse $\fou_{\tn}^{-1}: \ese(\zn)\rightarrow C^{\infty}(\tn)$ is given by 
\[\forall x\in \tn,\,\, \phi(x)=\sum\limits_{\xi\in\zn}e^{i2\pi x\cdot\xi}\widehat{\phi}(\xi).\]
Thus, for every $h\in\ese(\zn)$ we have
\[\forall x\in \tn,\,\,(\fou_{\tn}^{-1}h)(x)=\sum\limits_{\xi\in\zn}e^{i2\pi x\cdot\xi}h(\xi).\] \end{remark}
A fundamental notion for introducing the H\"ormander classes on the torus is the one of difference operators. It will allow us to ``differentiate'' symbols with respect to the discrete variable $\xi\in \mathbb{Z}.$ We record this notion as follows.
\begin{defn} {\bf (Finite differences $\Delta_{\xi}^{\alpha}$).} Let $\sigma:\zn\rightarrow \Co$ and $1\leq i,j\leq n.$ Let $\delta_j\in \ene_0^n$ be defined  by
\begin{equation}
 (\delta_j)_i:=\left\{
\begin{array}{rl}
1,& \mbox{ if } i=j\\
0,& \mbox{ if } i\neq j.
\end{array} \right.
\label{krona}\end{equation}

We define the {\em forward partial difference operator}   $\Delta_{\xi_j}$ by
\[\forall\xi\in \zn,\,\,\Delta_{\xi_j}\sigma(\xi)=\sigma(\xi+\delta_j)-\sigma(\xi),\] 
and for $\alpha\in \ene_0^n$ define
\[\Delta_{\xi}^{\alpha}:=\Delta_{\xi}^{\alpha_1}\cdots\Delta_{\xi}^{\alpha_n}.\]
\end{defn}

\begin{proposition}\label{combal} {\bf (Formulae for  $\Delta_{\xi}^{\alpha}$)} Let $\phi:\zn\rightarrow \Co$. We have 
\beq \forall\xi\in \zn,\,\forall\alpha\in \mathbb{N}_0^n,\,\,  \Delta_{\xi}^{\alpha}\phi(\xi)\,=\, \sum\limits_{\beta\leq\alpha}(-1)^{\lvert\alpha-\beta\rvert}{\alpha\choose\beta} \phi(\xi+\beta).\label{cohr}\eq
\end{proposition}
We now recall the definition of the  toroidal symbol classes. We  define the discrete Japanese bracket $\jpxi :=(1+\lvert\xi \rvert^2)^\half$ for $\xi\in\zn$, with $\lvert\xi\rvert^2=\xi_1^2+\cdots +\xi_n^2,$ denoting the Euclidean norm.
\begin{defn} Let $m\in\ar$, $0\leq \delta, \rho\leq 1.$ Then the {\em toroidal symbol class} $S_{\rho, \delta}^m(\mathbb{T}^n\times \zet^n)$ consists of those functions $a:=a(x, \xi)$ which are smooth in $x\in\tn,$ for all $\xi\in\zn$, and which satisfy the {\em toroidal symbol inequalities} 
\begin{equation}
    \lvert\Delta_{\xi}^{\alpha}\partial_x^{\beta}a(x,\xi)\rvert\leq C_{\alpha\beta}\jpxi^{m-\rho \lvert \alpha \rvert+\delta \lvert \beta\rvert}
\end{equation}
for every $x\in\tn, \xi\in\zn$ and for all $\alpha, \beta\in \ene_0^n$.
\end{defn} 
\begin{remark} The family of seminorms  \begin{equation}\label{InIC}
      p_{\alpha,\beta,\rho,\delta,m}(a):= \sup_{(x,\xi)\in \tn\times \zn }\jpxi^{\rho\lvert\alpha\rvert-\delta\lvert\beta\rvert-m}\lvert\partial_{x}^{\beta} \Delta_{\xi}^{\alpha}a(x,\xi)\rvert <\infty,
   \end{equation}defines a Fr\'echet structure on every toroidal class $S_{\rho, \delta}^m(\mathbb{T}^n\times \zet^n).$
\end{remark}
\begin{defn} Let $m\in\ar$ and let $0\leq \delta, \rho\leq 1.$
    The toroidal quantisation associated to a symbol $a\in S_{\rho, \delta}^m(\mathbb{T}^n\times \zet^n)$ is the densely defined  operator
    \begin{equation}\label{tro:quan}
      A=a(x,D):=\textnormal{Op}(a):C^\infty(\tn)\rightarrow C^\infty(\tn),  
    \end{equation}
    defined by 
\[ \,\,\,\,  \,\,\,\,\,\,\,\,\,\,\,\,\,\,\,\,Af(x)=\sum\limits_{\xi\in\zn}e^{2\pi ix\cdot \xi}a(x,\xi)\widehat{f}(\xi) ,\,\,\, f\in C^\infty(\tn),\,x\in \tn.\]
\end{defn}
\begin{remark}
    Note that if one has a continuous linear operator $A:C^\infty(\tn)\rightarrow C^\infty(\tn) $, its symbol $a:=a(x,\xi)$ can be recovered from the following formula
\beq a(x,\xi)=e^{-i2\pi x\cdot\xi}Ae_{\xi}(x),\eq
where for every $\xi\in\zn$, $e_{\xi}(x)=e^{i2\pi x\cdot\xi},$  $x\in\tn,$ is the canonical trigonometric polynomial.
\end{remark}
\begin{defn}The family of pseudo-differential operators defined by the toroidal quantisation formula in \eqref{tro:quan} corresponding to the class of symbols $a\in S_{\rho, \delta}^m(\mathbb{T}^n\times \zet^n)$ will be denoted by $\Op S_{\rho,\delta}^{m}(\mathbb{T}^n\times \zet^n)$.     
\end{defn}
The toroidal  quantisation has been extensively analysed in \cite{rt:torus,rt:book} for the general case of $\Tn$ and on compact Lie groups. For the toroidal H\"ormander class of order $m\in \mathbb{R},$ one has  $\Psi^m(\mathbb{T}^n,\textnormal{loc})=\{\sigma(x,D):\sigma\in S^m(\mathbb{T}^n\times \zet^n)\}$ (cf. \cite{rt:torus,rt:book}) for the class of symbols $\Psi^m(\mathbb{T}^n,\textnormal{loc})$ defined by local coordinate systems. 
\begin{remark} The family of seminorms  \begin{equation}\label{InIC}
      p_{\alpha,\beta,\rho,\delta,m}(A):= \sup_{(x,\xi)\in \tn\times \zn }\jpxi^{\rho\lvert\alpha\rvert-\delta\lvert\beta\rvert-m}\lvert\partial_{x}^{\beta} \Delta_{\xi}^{\alpha}a(x,\xi)\rvert <\infty,\,A=\textnormal{Op}(a),
   \end{equation}defines a Fr\'echet structure on every class $\textnormal{Op}(S_{\rho, \delta}^m(\mathbb{T}^n\times \zet^n)).$
\end{remark}
The following proposition gives the equivalence between  Euclidean H\"ormander's symbols and toroidal symbols. It corresponds to Corollary 4.6.13 in \cite{rt:torus,rt:book}.

\begin{proposition}\label{prequiv} For $m\in\ar$, $0\leq \delta\leq 1$ and $0<\rho\leq 1$ we have 
 \[\Psi^{m}_{\rho,\delta}:=\Op S_{\rho,\delta}^{m}(\mathbb{T}^n\times \ar^n)=\Op S_{\rho,\delta}^{m}(\mathbb{T}^n\times \zet^n),\]
 i.e., classes of $1$-periodic pseudo-differential operators  with Euclidean (H\"ormander's) symbols in $\Op S_{\rho,\delta}^{m}(\mathbb{T}^n\times \ar^n)$ and toroidal symbols in $\Op S_{\rho,\delta}^{m}(\mathbb{T}^n\times \zet^n)$ coincide.
\end{proposition}
In the next theorem we describe some fundamental properties of the  calculus of pseudo-differential operators on the torus.
\begin{theorem}\label{calculus} Let $0\leqslant \delta<\rho\leqslant 1,$ and let  $m\in \mathbb{R}.$ Then:
\begin{itemize}
    \item [-] the mapping $A\mapsto A^{*}:\Psi^{m }_{\rho,\delta}\rightarrow \Psi^{m }_{\rho,\delta},$ that assigns to each operator $A$ its formal adjoint $A^*,$ is a continuous linear mapping between Fr\'echet spaces and  the  symbol $\sigma_{A^*}(x,\xi)$ of $A^*,$  satisfies the asymptotic expansion,
 \begin{equation*}
    \sigma_{A^{*}}(x,\xi)\sim \sum_{\alpha}\Delta_{\xi}^\alpha\partial_{x}^{\alpha} (\overline{\sigma_{A}(x,\xi)}),\,(x,\xi)\in \tn\times \zn.
 \end{equation*} This means that, for every $N\in \mathbb{N},$ and for all $\ell\in \mathbb{N},$
\begin{align*}
    \Delta_{\xi}^{\alpha_\ell}\partial_{x}^{\beta}\left(\sigma_{A^{*}}(x,\xi)-\sum_{\lvert\alpha\rvert\leqslant N}\Delta_{\xi}^\alpha\partial_{x}^{\alpha} (\overline{\sigma_{A}(x,\xi)}) \right)\\
    \hspace{2cm}\in {S}^{m-(\rho-\delta)(N+1)-\rho\ell+\delta\lvert\beta\rvert}_{\rho,\delta}(\tn\times \zn),
\end{align*} where $\alpha_\ell\in \mathbb{N}_0^n$ is such that $\lvert\alpha_\ell\rvert=\ell.$
\item [-] The mapping $(A_1,A_2)\mapsto A_1\circ A_2: \Psi^{m_1}_{\rho,\delta}\times \Psi^{m_2}_{\rho,\delta}\rightarrow \Psi^{m_1+m_2}_{\rho,\delta}$ is a continuous bilinear mapping between Fr\'echet spaces, and the symbol of $A=A_{1}\circ A_2$ satisfies the asymptotic formula
\begin{equation*}
    \sigma_A(x,\xi)\sim \sum_{\alpha}(\Delta_{\xi}^\alpha\sigma_{A_{1}}(x,\xi))(\partial_{x}^{\alpha} \sigma_{A_2}(x,\xi)),
\end{equation*}which, in particular, means that, for every $N\in \mathbb{N},$ and for all $\ell \in\mathbb{N},$
\begin{align*}
    &\Delta_{\xi}^{\alpha_\ell}\partial_{x}^{\beta}\left(\sigma_A(x,\xi)-\sum_{\lvert\alpha\rvert\leqslant N}  (\Delta_{\xi}^\alpha\sigma_{{A_1}}(x,\xi))(\partial_{x}^{\alpha} \sigma{A_2}(x,\xi))  \right)\\
    &\hspace{2cm}\in {S}^{m_1+m_2-(\rho-\delta)(N+1)-\rho\ell+\delta\lvert\beta\rvert}_{\rho,\delta}(\tn\times \zn),
\end{align*}for all  $\alpha_\ell \in \mathbb{N}_0^n$ with $\lvert\alpha_\ell\rvert=\ell.$
\item [-] For  $0\leqslant \delta\leqslant \rho\leqslant    1,$    $\delta\neq 1,$ let us consider a continuous linear operator $A:C^\infty(\tn)\rightarrow\mathscr{D}'(\tn)$ with symbol  $\sigma\in {S}^{0}_{\rho,\delta}(\tn\times \zn)$. Then $A$ extends to a bounded operator from $L^2(\tn)$ to  $L^2(\tn).$ 
\end{itemize}
\end{theorem} 
\begin{remark}\label{Remark:Sob}
    We now recall the definition of the  Sobolev space of order $s\in\er$ on the torus.
 For $u\in\mathcal{D}'(\To^n)$ we define the norm $\|\cdot\|_{H^s(\To^n)}$ by
\[\|u\|_{H^s(\To^n)}:=\left(\sum\limits_{\xi\in\zet^n}(1+\lvert\xi\rvert^2)^s\lvert\hat{u}(\xi)\rvert^2\right)^{1/2}.\] 
 The Sobolev space $H^s(\To^n)$, is the space of the $1$-periodic distributions $u$ such that  $\|u\|_{H^s(\To^n)}<\infty.$ 
\end{remark}

\subsection{Matrix-valued classes on the torus}\label{matrix:values:classes}

In order to study  our problem \eqref{probei2gax}, we will reduce it to a first-order system with respect to the time variable. For this reduced system, we will require the notion  of a matrix-valued symbolic calculus on the torus which we introduce it as follows:
\begin{defn} Let $\ell\in \mathbb{N}.$ An $\ell\times \ell$-matrix symbol
\begin{equation}
    \mathfrak{a}:=[a_{ij}]_{1\leq i,j\leq \ell },\,\,a_{ij}:\tn\times \zn\rightarrow \mathbb{C},
\end{equation}
belongs to the matrix-valued class $S^m_{\ell\times\ell;\rho,\delta}(\tn\times\zn)$ if each one of its entries $a_{ij}$ is a symbol in $S^m_{\rho,\delta}(\tn\times\zn)$. The $\ell\times \ell$-matrix-valued pseudo-differential operator $\mathfrak{A}:=\textnormal{Op}(\mathfrak{a})$ associated to the symbol $\mathfrak{a},$ is defined via
\begin{equation}
   \textnormal{Op}( \mathfrak{a}):=[\textnormal{Op}(a_{ij})]_{1\leq i,j\leq \ell },\,\,\textnormal{Op}(a_{ij})\in \Psi^{m}_{\rho,\delta}(\tn\times \zn).
\end{equation}
\begin{remark}We will use the notation
\[\Psi^m_{\ell\times \ell;\rho,\delta}(\tn\times \zn):=\textnormal{Op}(S^m_{\ell\times\ell;\rho,\delta}(\tn\times\zn)),\] for the class of operators $\textnormal{Op}( \mathfrak{a})$ with a symbol $\mathfrak{a}$ belonging to the matrix-valued class $S^m_{\ell\times\ell;\rho,\delta}(\tn\times\zn).$    
\end{remark}
\end{defn}
\begin{remark} The family of seminorms  \begin{equation}\label{InIC:2}
      p_{\alpha,\beta,\rho,\delta,m}(\mathfrak{a}):= \sup_{(x,\xi)\in \tn\times \zn }\sup_{1\leq i,j\leq \ell}\jpxi^{\rho\lvert\alpha\rvert-\delta\lvert\beta\rvert-m}\lvert\partial_{x}^{\beta} \Delta_{\xi}^{\alpha}a_{ij}(x,\xi)\rvert <\infty,
   \end{equation}defines a Fr\'echet structure on every class $S_{\ell\times \ell;\rho, \delta}^m(\mathbb{T}^n\times \zet^n).$
Moreover,  the family of seminorms  \begin{equation}\label{InIC:3}      p_{\alpha,\beta,\rho,\delta,m}(\textnormal{Op}(\mathfrak{a})):= p_{\alpha,\beta,\rho,\delta,m}(\mathfrak{a}),\,A=\textnormal{Op}(a),
   \end{equation}defines a Fr\'echet structure on every class $\textnormal{Op}(S_{\ell\times \ell;\rho, \delta}^m(\mathbb{T}^n\times \zet^n)).$
\end{remark}
\begin{remark} The basic properties of the calculus for $S^m_{\rho,\delta}(\tn\times\zn)$ also hold for the matrix-valued classes $S^m_{\ell\times\ell;\rho,\delta}(\tn\times\zn)$. Indeed, the classes of symbols $S^m_{\rho,\delta}(\tn\times\zn),$ $m\in \mathbb{R},$ are closed under the addition
\[\mathfrak{a}_1+\mathfrak{a}_2=[a_{1,ij}+a_{2,ij}]_{i,j=1}^\ell,\,\mathfrak{a}_k=[a_{k,ij}]_{i,j=1}^\ell,\, k=1,2,\]
of matrix-valued symbols. Indeed, note that
\begin{equation}
S^{m_1}_{\ell\times\ell;\rho,\delta}(\tn\times\zn)\times S^{m_2}_{\ell\times\ell;\rho,\delta}(\tn\times\zn) \ni (\mathfrak{a}_1,\mathfrak{a}_2)\mapsto   \mathfrak{a}_1+\mathfrak{a}_2\in S^{m_1+m_2}_{\ell\times\ell;\rho,\delta}(\tn\times\zn).\,
\end{equation}Is clear that, in this case, 
\[\textnormal{Op}(\mathfrak{a}_1+\mathfrak{a}_2)=\textnormal{Op}(\mathfrak{a}_1)+\textnormal{Op}(\mathfrak{a}_2)\in \Psi^{m_1+m_2}_{\ell\times\ell;\rho,\delta}(\tn\times\zn). \] Also, observe that if we define on $L^{2}(\tn;\mathbb{C}^{\ell})$ the inner product
\[ (u,v)_{L^2(\tn;\mathbb{C}^{\ell})}=\sum_{i=1}^{\ell}(u_{i},v_{i})_{L^2(\tn)},\,u:=(u_1,\cdots,u_{\ell}),\,v:=(v_1,\cdots,v_{\ell})\in L^{2}(\tn;\mathbb{C}^{\ell}) , \]
  we have that  
\begin{align*}
    (\textnormal{Op}(\mathfrak{a})u,v)_{L^2(\tn;\mathbb{C}^{\ell})}&=\sum_{i=1}^{\ell}((\textnormal{Op}(\mathfrak{a})u)_{i},v_{i})_{L^2(\tn)}=\sum_{i=1}^\ell\sum_{k=1}^\ell (\textnormal{Op}(\mathfrak{a})_{ik}u_k,v_{i})_{L^2(\tn)}\\
    &=\sum_{k=1}^\ell \sum_{i=1}^\ell(u_k,\textnormal{Op}(\mathfrak{a})_{ik}^{*}v_{i})_{L^2(\tn)}.
\end{align*}So, defining
\begin{equation}
 \textnormal{Op}(\mathfrak{a}^{*})_{ki}:=   \textnormal{Op}(\mathfrak{a})_{ik}^{*}
\end{equation} the formal adjoint $\textnormal{Op}(\mathfrak{a})^*$ of $\textnormal{Op}(\mathfrak{a})$ is determined by the entries
\begin{equation}
    \textnormal{Op}(\mathfrak{a})^{*}=[\textnormal{Op}(\mathfrak{a}^{*})_{ki}]_{k,i=1}^{\ell}.
\end{equation}Indeed, observe that
\begin{align*}
     (\textnormal{Op}(\mathfrak{a})u,v)_{L^2(\tn;\mathbb{C}^{\ell})}&=\sum_{k=1}^\ell \sum_{i=1}^\ell(u_k,\textnormal{Op}(\mathfrak{a})_{ik}^{*}v_{i})_{L^2(\tn)}=\sum_{k=1}^\ell \sum_{i=1}^\ell(u_k,\textnormal{Op}(\mathfrak{a}^{*})_{ki}v_{i})_{L^2(\tn)}\\
     &=\sum_{k=1}^\ell(u_k,(\textnormal{Op}(\mathfrak{a}^{*})v)_k)_{L^2(\tn;\mathbb{C}^{\ell})}\\
     &=(u,\textnormal{Op}(\mathfrak{a}^*)v)_{L^2(\tn;\mathbb{C}^{\ell})}.
\end{align*}The previous analysis also proves that if $\textnormal{Op}(\mathfrak{a})\in \Psi^{m}_{\ell\times \ell;\rho,\delta}(\tn\times \zn)$ then $\textnormal{Op}(\mathfrak{a})^*\in \Psi^{m}_{\ell\times \ell;\rho,\delta}(\tn\times \zn).$  
\end{remark}
In view of the previous remark and as a direct consequence of Theorem \ref{calculus} the calculus for matrix-valued pseudo-differential operators can be summarised in the following result.
\begin{theorem}\label{calculus:II} Let $0\leqslant \delta<\rho\leqslant 1,$ and for every $m\in \mathbb{R}.$ Then:
\begin{itemize}
    \item [-] the mapping $A\mapsto A^{*}:\Psi^{m }_{\ell\times \ell;\rho,\delta}(\tn\times \zn)\rightarrow \Psi^{m }_{\ell\times \ell;\rho,\delta}(\tn\times \zn),$ that assigns to each operator $A$ its formal adjoint $A^*,$ is a continuous linear mapping between Fr\'echet spaces.
\item [-] The mapping $(A_1,A_2)\mapsto A_1\circ A_2: \Psi^{m_1}_{\ell\times \ell;\rho,\delta}(\tn\times \zn)\times \Psi^{m_2}_{\ell\times \ell;\rho,\delta}(\tn\times \zn)\rightarrow \Psi^{m_1+m_2}_{\ell\times \ell;\rho,\delta}(\tn\times \zn)$ is a continuous bilinear mapping between Fr\'echet spaces.
\end{itemize}
\end{theorem}
\begin{remark}\label{rem:lambda} Consider the positive Laplacian $\mathcal{L}_{\tn}=-\sum_{i=1}^n\partial_{x_i}^2$ on the torus $\tn.$ In view of the functional calculus for elliptic operators one has that
\begin{equation}
    \Lambda^s:=(1+\mathcal{L}_{\tn})^{\frac{s}{2}}\in \Psi^{s}_{1,0}(\tn\times \zn),\,s\in \mathbb{R}.
\end{equation}Note that on $C^{\infty}(\tn, \mathbb{C}^\ell)$ the operator
\begin{equation}
    \Lambda^s_{\ell\times \ell}:=[\Lambda^s\delta_{ij}]_{i,j=1}^\ell,
\end{equation} where $\delta_{ij}$ is the Delta-Kronecker, defines an operator in the class $\Psi^{s}_{\ell\times\ell;1,0}(\tn\times \zn).$ Since, for any $u=(u_1,\cdots,u_{\ell})\in C^{\infty}(\tn, \mathbb{C}^\ell) $ one has that
\begin{align*}
    \Lambda^s_{\ell\times \ell}u=(\Lambda^s u_1,\cdots,\Lambda^s u_{\ell}),
\end{align*}in order to simplify the notation we will always write $\Lambda^s:=\Lambda^s_{\ell\times \ell}$ omitting the sub-index $\ell\times \ell.$     
\end{remark} Now, the simplified notation above allows us to introduce   the vector-valued Sobolev space $H^s(\To^n,\mathbb{C}^{\ell} )$, defined by the $1$-periodic distributions $u$  such that  
\[\|u\|_{H^s(\To^n,\mathbb{C}^{\ell})}=\Vert \Lambda^s u \Vert_{L^2(\tn,\mathbb{C}^{\ell})}=\left(\sum_{i=1}^{\ell}\Vert \Lambda^s u_i \Vert_{L^2(\tn)}^{2}\right)^{\frac{1}{2}}<\infty.\] 

\section[Pseudo-differential hyperbolic equations on the torus]{Pseudo-differential hyperbolic equations on the torus}\label{Sect:3} 
In this section we establish the well-posedness for the problem \eqref{probei2gax}.
 We first recall the notions of ellipticity and strong ellipticity in the setting of the 
 $S_{\rho, \delta}^{m}(\mathbb{T}^n\times \zet^n)$-classes. We then  obtain energy estimates and the corresponding consequences on the well-posedness for first order systems and its consequence on Fractional hyperbolic equations. Some aspects on the the fractional Laplacian are briefly reviewed.  For a comprehensive study  on this theory we refer to  \cite{rt:torus,rt:book} and for a more detailed discussion  on the fractional Laplacian on the torus we refer to \cite{ron:t}.\\

The classical notion of ellipticity for pseudo-differential operators on $\arn$ extends into the toroidal setting in an analogous way. We point out that some further restrictions on $\rho, \delta$ have to be imposed in order to obtain an useful definition. 

\begin{defn} Let $m\in\er$,  $0\leq \delta< \rho\leq 1$ and  $\sigma$ be a symbol in $\sigma\in S_{\rho, \delta}^m(\To^n\times\zet^n)$. 
\begin{itemize}
    \item We say that the corresponding  pseudo-differential operator $\sigma(x,D)$ is {\em elliptic of order }$m$, if $\sigma$ satisfies
\beq \exists n_0\in \mathbb{N},\,\forall (x,\xi)\in \To^n\times\zet^n:\,\, \lvert\xi\rvert\geq n_0\implies \lvert\sigma(x,\xi)\rvert\geq C_0\langle \xi\rangle^m ,\eq for some constants $n_0, C_0>0$.
\item We will also say that $\sigma(x,D)$ is a {\em strongly elliptic operator} 
if $\sigma$ satisfies:
\beq \exists n_0\in \mathbb{N},\, \forall (x,\xi)\in \To^n\times\zet^n:\,\, \lvert\xi\rvert\geq n_0\implies {\mbox Re }\, \sigma(x,\xi)\geq C_0\langle \xi\rangle^m ,\eq  
for some constants $n_0, C_0>0$.
\end{itemize}
\end{defn}

\begin{remark}
We now specialize into the case of the fractional Laplacian.  We first state a basic property of this interesting operator on the torus  that distinguish it from the case of $\arn$. As a consequence of this fact and the above theorem we obtain some corollaries for the  fractional diffusion.

Let  $\nu$ be a strictly positive real number, the fractional Laplacian $(-\Delta)^{\frac{\nu}{2}}$ on the torus $\tn$ is defined as the  Fourier multiplier corresponding to  $\lvert\xi\vert^{\nu}$, that is 
\beq \widehat{(-\Delta)^{\frac{\nu}{2}}u}(\xi)=(2\pi)^{\nu}\lvert\xi\vert^{\nu}\widehat{u}(\xi),\eq 
 for every $\xi\in\zn$ and $u\in C^{\infty}(\tn)$. The Fourier inversion formula on the torus also allows the integral representation in \eqref{toroidal:fraclap} for  $(-\Delta)^{\frac{\nu}{2}}$. Other alternative definitions are possible as in the case of $\Rn$. We refer the reader to \cite{mk:aw} for a recent review on the most common ones and the equivalence between them on $\Rn$. We point out that on $\Rn$, the fractional Laplacian $(-\Delta)^{\frac{\nu}{2}}$ has not a symbol in a class of pseudo-differential operators unless $\frac{\nu}{2}$ be an integer, however in our toroidal setting  the fractional Laplacian will be a pseudo-differential operator for every $\nu>0.$  We first state a mild lemma on the torus  clarifying such property.
 \end{remark}

\begin{lemma}\label{flle} Let $\nu>0.$ Then $(-\Delta)^{\frac{\nu}{2}}\in \Op S_{1,0}^{\nu}(\mathbb{T}^n\times \zet^n)$. 
\end{lemma}
 \begin{proof} We should show that the function $\sigma:\zn\rightarrow\ar$ defined by $\sigma(\xi)=\lvert\xi\vert^{\nu}$ belongs to  $S_{1,0}^{\nu}(\mathbb{T}^n\times \zet^n)$. We choose a function $\chi\in C^{\infty}(\Rn)\,$ such that
 \begin{equation}
 \chi(\xi):=\left\{
\begin{array}{rl}
1,& \mbox{ if } \lvert\xi\rvert\geq 1,\\
0,& \mbox{ if } \lvert\xi\rvert\leq \half.
\end{array} \right.
\label{ht56e}\end{equation}
We observe that  $\widetilde{\sigma}(\xi):=\chi(\xi)\lvert\xi\rvert^{\nu}$ is smooth on $\Rn$ and  $\widetilde{\sigma}\in S_{1,0}^{\nu}(\mathbb{T}^n\times \Rn)$. Hence $\widetilde{\sigma}(x,D)\in \Op S_{1,0}^{\nu}(\mathbb{T}^n\times \Rn)$ and by  Proposition \ref{prequiv} we obtain that $\widetilde{\sigma}(x,D)\in \Op S_{1,0}^{\nu}(\mathbb{T}^n\times \zn)$. 
On the other hand we have that $\chi(\xi)\lvert\xi\rvert^{\nu}=\lvert\xi\rvert^{\nu}=\sigma(\xi)$ for all $\xi\in\zn$, and  therefore $\sigma\in S_{1,0}^{\nu}(\mathbb{T}^n\times \zet^n)$.
\end{proof}
\begin{remark}
The fractional Laplacian  $(-\Delta)^{\frac{\nu}{2}}$ is an element of the class  $\Op S_{1,0}^{\nu}(\mathbb{T}^n\times \zet^n)$ for any $\nu>0$. Then it is clear that it is also a  strongly elliptic operator of order $\nu$. This property is  an advantage over the classical case of the fractional Laplacian on $\arn$ where one can only refer to it  as a pseudo-differential operator when $\nu$ is an even integer. 
\end{remark}
In the following theorems we will write $K(t,x,\xi)\in S^{m}_{\ell\times\ell}(\tn\times\zn)$, which should be understood in the sense that for each $t\in\ar$ fixed, 
 $K(t,\cdot,\cdot)\in S^{m}_{\ell\times\ell;1, 0}(\tn\times\zn)$. We will also require of the  spaces  
 $$H^{+\infty}:=C^{\infty}( \tn)=\bigcap\limits_{s\in\ar}H^s \text{ and } H^{-\infty}:=\mathscr{D}'(\tn)=\bigcup\limits_{s\in\ar}H^s.$$
 
 The following theorem establishes a fundamental energy estimate is the toroidal version within the setting of the pseudo-differential calculus as introduced in \cite{rt:torus,rt:book}.
 
 \begin{theorem}\label{order1} Let $m>0$ and $K(t,x,\xi)\in S^m_{\ell\times\ell;\rho,\delta}(\tn\times\zn)$ depending smoothly on $t$. Assume that $K^*(t,x,D_x)+K(t,x,D_x)\in OpS^{0}_{\ell\times\ell;\rho,\delta}(\tn\times\zn)$, where for $t$ fixed, $K(t,x,D_x)=K(t)$ denotes the pseudo-differential operator corresponding to $K(t,x,\xi)$. Let $s\in\ar$, $T>0$. Let  
 $$v\in C([0,T], H^{s+1}(\tn;\mathbb{C}^{\ell}))\bigcap C^1([0,T], H^{s}(\tn;\mathbb{C}^{\ell}))$$ 
 and $Q:=\partial_t-K(t).$  
Then $v\in C^\infty(\tn,\mathbb{C}^{\ell})$ satisfies
\beq  \|v(t)\|_{H^s(\tn,\mathbb{C}^{\ell})}^2\leq e^{Ct}\left(\|v(0)\|_{H^s(\tn,\mathbb{C}^{\ell})}^2+\int\limits_0^t\|\omega(\tau)\|^2_{H ^s(\tn,\mathbb{C}^{\ell})}d\tau\right)\label{eineq:1a}  
\eq
for all $t\in [0,T].$ Moreover, we can replace $v(0)$ by $v(T)$ on the right-hand side of \eqref{eineq:1a}. The same conclusion holds
 for the operator $Q^*$.
\end{theorem}
\begin{proof} For any $s\in \mathbb{R},$ let us write $H^s:=H^s(\tn,\mathbb{C}^{\ell}).$ We set $\Lambda^s$ be the Fourier multiplier given by $\Lambda(\xi)=(1+ \lvert\xi\rvert^2)^{\frac{s}{2}}$. Let us follow the notation in Remark \ref{rem:lambda}.   It is clear that $u$ belongs to $H^s(\tn,\mathbb{C}^{\ell})$ if and only if $\Lambda^su\in L^2(\tn,\mathbb{C}^{\ell})$. We assume that $v\in C([0,T], H^{s+1})\bigcap C^1([0,T], H^{s})$ and write $\omega=Qv$.
 Since $\partial_tv=(\partial_t-K(t))v+K(t)v=Qv+K(t)v=\omega +K(t)v$, we observe that
\begin{align}
\frac{d}{dt}\|v(t)\|_{H^s}^2 &=\frac{d}{dt}\langle\Lambda^sv,\Lambda^s v\rangle\nonumber\\ 
&=2Re \langle\Lambda^sv_t,\Lambda^s v\rangle\nonumber\\
&=2Re \langle\Lambda^s(K(t)v+\omega),\Lambda^s v\rangle\nonumber\\
&=2Re \langle\Lambda^sK(t)v,\Lambda^s v\rangle\nonumber\\
&-2Re \langle K(t)\Lambda^sv,\Lambda^s v\rangle\nonumber\\
&+2Re \langle K(t)\Lambda^sv,\Lambda^s v\rangle\nonumber\\
&+2Re \langle \Lambda^s\omega,\Lambda^s v\rangle\nonumber\\
&=2Re \langle [\Lambda^s,K(t)]v,\Lambda^s v\rangle\nonumber\\
&+2Re \langle K(t)\Lambda^sv,\Lambda^s v\rangle\label{ene2}\\
&+2Re \langle \Lambda^s\omega,\Lambda^s v\rangle.\nonumber
\end{align}
We note that the term (\ref{ene2}) can be written in the following way 
\begin{align*}
2Re \langle K(t)\Lambda^sv,\Lambda^s v\rangle &=\langle K(t)\Lambda^sv,\Lambda^s v\rangle + \overline{\langle K(t)\Lambda ^sv,\Lambda^s v\rangle}\\
&=\langle K(t)\Lambda^sv,\Lambda^s v\rangle+\langle\Lambda ^sv,K(t)\Lambda ^s v\rangle\\
&=\langle K(t)\Lambda^sv,\Lambda^s v\rangle+\langle K(t)^*\Lambda ^sv,\Lambda ^s v\rangle\\
&=\langle (K(t)+K(t)^*)\Lambda^sv,\Lambda^s v\rangle.
\end{align*}
Now,  we get 
 $A(t)= [\Lambda^s,K(t)]\in OpS^s_{\ell\times \ell;\rho,\delta}(\tn\times\zn)$. Indeed, observe that
 \[A(t)= [\Lambda^s,K(t)]=[A(t)_{ij}]_{i,j=1}^{\ell}\]
 where
 \begin{align*}
     A(t)_{ij}&=([\Lambda^s,K(t)])_{ij}\\
     &=[\Lambda^sK(t)]_{ij}-[K(t)\Lambda^s]_{ij}\\
     &=\sum_{k=1}^{\ell}\Lambda^s\delta_{ik}K(t)_{kj}- \sum_{k=1}^{\ell}K(t)_{ik}\Lambda^s\delta_{kj}\\
     &=\Lambda^sK(t)_{ij}-K(t)_{ij}\Lambda^s\\
     &=[\Lambda^s,K(t)_{ij}].
 \end{align*}The commutator properties of the scalar-valued calculus on the torus implies that any entry $[\Lambda^s,K(t)_{ij}]$ belong to the class $OpS^{s+1-1}_{\rho,\delta}(\tn\times\zn)=OpS^s_{\rho,\delta}(\tn\times\zn).$ Consequently, we have that $A(t)= [\Lambda^s,K(t)]\in OpS^s_{\ell\times \ell;\rho,\delta}(\tn\times\zn)$ as claimed.   
 Since we also have $K(t)+K(t)^*\in OpS^{0}_{\ell\times\ell;\rho,\delta}(\tn\times\zn)$, it follows that
\[\frac{d}{dt}\|v(t)\|_{H^s}^2 \leq\]
\begin{align*}
&\leq \|A(t)v\|_{L^2}\|v\|_{H^s}+C_1\|v\|^2_{H^s}+C_2\|\omega\|_{H^s}\|v\|_{H^s}\\
&\leq C\|v\|_{H^s}\|v\|_{H^s}+C_1\|v\|^2_{H^s}+C_2\|\omega\|_{H^s}\|v\|_{H ^s}\\
&\leq C\|v\|^2_{H^s}+C\|\omega\|^2_{H^s}.
\end{align*}
An application of the Gronwall inequality gives us the energy inequality
\beq\label{enineq}\|v(t)\|_{H^s}^2\leq e^{Ct}\left(\|v(0)\|_{H^s}^2+\int\limits_0^t\|\omega(\tau)\|^2_{H ^s}d\tau\right).\eq
 We can also prove an analogous estimate with $v(T)$ instead of $v(0)$ on the right-hand side of the inequality \eqref{enineq}.
  The conclusion for $Q^*$ follows analogously.
\end{proof}

We now obtain a consequence regarding the existence, uniqueness and regularity as an application of the above estimates.
\begin{theorem}\label{order1a} Let $m>0$ and $K(t,x,\xi)\in S^m_{\ell\times\ell}(\tn\times\zn)$ depending smoothly on $t$. Assume that $K^*(t,x,D_x)+K(t,x,D_x)\in OpS^{0}_{\ell\times\ell;\rho,\delta}(\tn\times\zn)$.  Let $s\in\ar$, $T>0$, $f\in H^s,\, \omega\in L^2([0,T], H^s)$. Then, there exists a unique
 $v\in C([0,T], H^s)$ such that
\begin{equation}
 \left\{
\begin{array}{rl}
{\displaystyle \partv} =&K(t)v+\omega, \,\, \mbox{ (in the sense of }\mathcal{D}'(]0,T[\times\Tn)\\
{\displaystyle v}(0)=&f.
\end{array} \right.
\label{probei2}\end{equation} 
  Moreover, the solution $v$ satisfies the energy estimate \eqref{eineq:1a}. If $\omega\in C^{\infty}([0,T],H^{+\infty})$ and $f\in H^{+\infty}$
   then $v\in C^{\infty}([0,T],H^{+\infty})$.
\end{theorem}
\begin{proof}
We will now prove the  existence of a solution $v$ of \eqref{probei2} in $C([0,T], H^s)$.  The proof is an adaptation of the corresponding part in the proof of Theorem 4.5 in \cite{chpi:book}. We write $Q=\frac{\partial}{\partial t}-K$ and we introduce the  space
 $E=\{\varphi\in C^{\infty}([0,T],H^{-\infty}): \varphi(T)=0\}$. We will see that we can define a linear form $\beta$ on $Q^*E$ by 
\[Q^*\varphi\rightarrow \beta(Q^*\varphi)=\int\limits_0^T(\omega(t,\cdot),\varphi(t,\cdot))dt+\frac{1}{i}(f,\varphi(0,\cdot)).\] 
We note that the energy estimate \eqref{eineq:1a} holds for $-s$ for the Cauchy problem \eqref{probei2} corresponding to the operator $Q^*$ with $\|v(T,\cdot)\|_{H^s}$ on the right hand side of \eqref{eineq:1a}. Thus, for $\varphi\in E$ we have
\[\|\varphi(t,\cdot)\|_{H^{-s}}^2\leq C\int\limits_0^T\|Q^*\varphi(t',\cdot)\|_{H^{-s}}^2dt' ,\,\,\,  t\in [0,T]\,\,,\]
 so that
$$\lvert\beta(Q^*\varphi)\rvert^2\leq C'\int\limits_0^T\|Q^*\varphi(t',\cdot)\|_{H^{-s}}^2dt'.$$
We deduce that $\beta$ is well defined and continuous with respect to the topology induced on $Q^*E$ by $L^2([0,T,], H^{-s})$. An application of the Hahn-Banach theorem implies the existence of an element $v\in (L^2([0,T,], H^{-s})'=L^2([0,T,], H^{s})$ such that
\beq (v,Q^*\varphi)=\int\limits_0^T(\omega(t,\cdot),\varphi(t,\cdot))dt+\frac{1}{i}(f,\varphi(0,\cdot))\label{exv}\eq
for all $\varphi\in E$. In particular, if $\varphi\in C^{\infty}(]0,T[\times\Tn)$, \eqref{exv} implies that $Qv=\omega$ in $\mathcal{D}'([0,T]\times\Tn))$. Thus
$\frac{\partial}{\partial t}v=Kv+\omega\in L^2([0,T,], H^{s-1}) $. An integration by parts with respect to $t$ in \eqref{exv} implies that
$(v(0,\cdot),\varphi(0,\cdot))=(f,\varphi(0,\cdot))$ for all $\varphi\in E$ and
consequently $v(0)=v(0,\cdot)=f$.   

Now, if $\omega\in C^{\infty}([0,T],H^{+\infty})$ and $f\in H^{+\infty}$, the above argument shows that $v\in C([0,T],H^{+\infty})$. Moreover, since $\frac{\partial}{\partial t}v=Kv+\omega$, one can deduce step by step that $v\in C^k([0,T],H^{+\infty})$ for all $k\geq 0$. Consequently $v\in C^{\infty}([0,T],H^{+\infty})$.

We will now prove that $v\in C([0,T],H^s)$ and that it satisfies the energy estimate \eqref{eineq:1a}. 
 Suppose we have sequences $(\omega_j)$ in $C^{\infty}([0,T]\times\Tn)$ and $(f_j)$ in $C_0^{\infty}(\Tn)$ such that $\omega_j\rightarrow \omega$ in $L^2([0,T,], H^{s})$ and $f_j\rightarrow f$ in  $H^{s}$. Let 
$v_j\in C^{\infty}([0,T],H^{+\infty})$ be the solution of $Qv_j=\omega_j, v_j(0,\cdot)=f_j$. The inequality \eqref{eineq:1a} applied to the $v_j-v_k$ shows that
 $v_j$ is a Cauchy sequence in $C([0,T],H^{s})$ so that $v_j\rightarrow\widetilde{v}$ in $C([0,T],H^{s})$. In the limit, we have $Q\widetilde{v}=\omega,\, \widetilde{v}(0,\cdot)=f$;  consequently, the uniqueness shows that $\widetilde{v}=v$. 

The corresponding inequality \eqref{eineq:1a} for $v$ is obtained passing to the limit in this inequality applied to $v_j$. In this way we conclude the proof of the Theorem. The uniqueness of the solution $v$ follows from the energy inequality \eqref{eineq:1a}. 
\end{proof}

We can now establish the well-posedness for the Cauchy problem \eqref{probei2gax} as a consequence of Theorem \ref{order1a}. We are ready to prove our main theorem.
\begin{proof}[Proof of Theorem \ref{main2:2}]
In order to apply  Theorem \ref{order1a}, we define  first  $$A=(I+P)^{\frac{1}{2}}.$$ Observe that
\begin{equation}\label{deco:mat}
    \partt  \underbrace{\partu}_{\rm v_2}=-P(x,D)A^{-1}\underbrace{Au}_{\rm v_1} ,
\end{equation}
with
$v_1=Au,\, v_2=\partu $.
In view of the functional calculus for elliptic operators we have that $A\in OpS^{\frac{\nu}{2}}(\Tn\times\zn).$ Hence, we have the matrix identity
\[
\partt\left[ {\begin{array}{c}
 v_1 \\
v_2\\
 \end{array} } \right]=\left[ {\begin{array}{cc}
 0 & A \\
-P(x,D)A^{-1} & 0 \\
 \end{array} } \right]\left[ {\begin{array}{c}
 v_1 \\
v_2\\
 \end{array} } \right]+\left[ {\begin{array}{c}
 0 \\
w\\
 \end{array} } \right].
\]
Indeed, the identity in the first component is reduced to the fact that
\begin{align*}
   \partt v_1=\partt Au= A\partt u=av_2,
\end{align*}where we have used that $A$ is independent of $t,$ and then the operators $A$ and $\partt$ commute. On the other hand, the identity in the second component is consequence of the fact that $\partt v_2=u_{tt},$ and the decomposition  in \eqref{deco:mat}.

It is clear that $K(t)\in OpS^{\frac{\nu}{2}}_{2\times2,\rho,\delta}(\Tn\times\zn)$, where $K(t)$ is the constant in $t,$ $2\times2$ matrix-valued operator:
\[
K(t) =
\left[ {\begin{array}{cc}
 0 & A \\
-P(x,D) A^{-1} & 0 \\
 \end{array} } \right]\,.
\]
On the other hand, observe that $K(t)+K(t)^*\in OpS_{2\times2}^{0}(\Tn\times\zn)$. 
To prove this, let us compute $K(t)^*$ as follows
\begin{align*}
  K(t)^* =
\left[ {\begin{array}{cc}
 0 & (-P(x,D) A^{-1})^* \\
A & 0 \\
 \end{array} } \right]\,.  
\end{align*}Since the operator $A^{-1}$ commutes with  $P$, we have that
$$ (-P(x,D) A^{-1})^*=-P(x,D) A^{-1},$$
in view of the positivity of $P(x,D).$ Note also that
\[ K(t)+K(t)^* =
\left[ {\begin{array}{cc}
 0 & A-P(x,D) A^{-1} \\
A-P(x,D) A^{-1} & 0 \\
 \end{array} } \right].\,\]
Indeed
\[ A^*A=I+P(x,D),\]
and 
\[ A^*=A^{-1}+P(x,D)A^{-1} =P(x,D)A^{-1}+R,\]
with 
$$R=A^{-1}\in OpS^{-\frac{\nu}{2}}_{\rho,\delta}(\tn\times\zn)\subset OpS^{0}_{\rho,\delta}(\tn\times\zn) .$$
Since $A$ is positive, we have that 
 \[A-P(x,D)A^{-1}=A^*-P(x,D)A^{-1}\in OpS^0(\Tn\times\zn),\]
 and so, we have proved that $K(t)+K(t)^*\in OpS_{2\times2;\rho,\delta}^{0}(\Tn\times\zn)$. Now, since $Af_0\in H^{s-\frac{\nu}{2}}$ and applying Theorem \ref{order1a} to 
\beq v(0)=f=\left[ {\begin{array}{c}
 Af_0 \\
f_1\\
 \end{array} } \right]\in H^{s-\frac{\nu}{2}},\,\,\omega=\left[ {\begin{array}{c}
 0 \\
w\\
 \end{array} } \right]\in L^2([0,T], H^{s-\frac{\nu}{2}}),
\label{eq:data1w}\eq
we obtain that $v\in C([0,T],H^{s-\frac{\nu}{2}})$. Since $u=A^{-1}v_1\in H^{s} $ we deduce that $u\in C([0,T],H^{s})$.
 The uniqueness of $u$ follows from the uniqueness of $v$ and the invertibility of $A$ since $u=A^{-1}v_1$. 
 The inequality \eqref{einw1:1c} is an immediate consequence of the inequality \eqref{eineq:1a} applied to the data \eqref{eq:data1w}. The last conclusion of the theorem also follows from the analogous part of Theorem \ref{order1a}.
\end{proof}
Observe that in view of Theorem \ref{main2:2} applied to the fractional Laplacian $P(x,D)=(-\Delta)^{\frac{\nu}{2}}$ we have the following energy estimate.

\begin{cor}\label{main2}  Let $\nu>0$, $T>0$. If $s\in\ar$, $f_0\in H^s,\,f_1\in H^{s-\frac{\nu}{2}}, \,  \omega\in L^2([0,T], H^{s-\frac{\nu}{2}})$. Then, there exists a unique
 solution $u\in C([0,T], H^s)$ of the Cauchy problem
 \begin{equation}
 \left\{
\begin{array}{rl}
{\displaystyle\partuu }\,\,\,\,=& -(-\Delta)^{\frac{\nu}{2}}u+w, \,\, \mbox{ (in the sense of }\mathcal{D}'(]0,T[\times\Tn))\\
{\displaystyle u}(0)=&f_0,\\
{\displaystyle\partu}(0)=& f_1.
\end{array} \right.
\label{probgg2}\end{equation}  
Moreover, the solution $u$ satisfies the following energy estimate 
\begin{equation}\label{einw1:1c}
    \Vert u(t)\Vert_{H^s}^2\leq Ce^{Ct}\left(\Vert f_0\Vert_{H^s}^2+\Vert f_1\Vert_{H^{s-\frac{\nu}{2}} }^2+\int\limits_0^t\Vert w(\tau)\Vert^2_{H^{s-\frac{\nu}{2}}}d\tau\right).
\end{equation}
Moreover, if $w\in C^{\infty}([0,T],H^{+\infty})$ and $f\in H^{+\infty}$
   then $u\in C^{\infty}([0,T],H^{+\infty})$.
\end{cor}




\end{document}